\documentclass[11pt,reqno,twoside]{amsart}
\usepackage{amssymb,amsmath,amsthm,soul,color,paralist}
\usepackage{t1enc}
\usepackage[cp1250]{inputenc}
\usepackage{a4,indentfirst,latexsym}
\usepackage{graphics}
\usepackage{mathrsfs}
\usepackage{cite,enumitem,graphicx}
\usepackage[colorlinks=true,urlcolor=blue,
citecolor=red,linkcolor=blue,linktocpage,pdfpagelabels,
bookmarksnumbered,bookmarksopen]{hyperref}
\usepackage[english]{babel}
\usepackage[top=2.50cm, bottom=2.50cm, left=2.72cm, right=2.72cm]{geometry}
\usepackage[metapost]{mfpic}
\usepackage[hyperpageref]{backref}
\usepackage[colorinlistoftodos]{todonotes}
\usepackage[normalem]{ulem}

\numberwithin{equation}{section}

\newcommand{\R}{\mathbb{R}}
\newcommand{\N}{\mathbb{N}}

\newcommand{\Z}{\mathbb{Z}}

\newcommand{\J}{\mathcal{J}}
\newcommand{\I}{\mathcal{I}}
\newcommand{\A}{\mathcal{S}_{2\pi}}
\newcommand{\e}{\varepsilon}
\newcommand{\2}{2^{*}_{s}}
\newcommand{\X}{\mathbb{X}^{s}_{2\pi}}
\newcommand{\h}{\mathbb{H}^{s}_{2\pi}}
\newcommand{\T}{\textup{Tr}}
\newcommand{\ri}{\rightarrow}

\DeclareMathOperator{\dive}{div}

\DeclareMathOperator{\supp}{supp}

\newtheorem{lem}{Lemma}
\newtheorem{prop}{Proposition}
\newtheorem{thm}{Theorem}

\newtheorem{defn}{Definition} 
\theoremstyle{remark}
\newtheorem{remark}{Remark}

\title[Periodic solutions for a class of fractional Kirchhoff problems]{Infinitely many periodic solutions for a class of fractional Kirchhoff problems}

\author[V. Ambrosio]{Vincenzo Ambrosio}
\address{Vincenzo Ambrosio\hfill\break\indent 
Dipartimento di Ingegneria Industriale e Scienze Matematiche \hfill\break\indent
Universit\`a Politecnica delle Marche\hfill\break\indent
Via Brecce Bianche, 12\hfill\break\indent
60131 Ancona (Italy)}
\email{v.ambrosio@univpm.it}

\keywords{periodic solutions; fractional Kirchhoff equation; variational methods; critical exponent}
\subjclass[2010]{34K13, 35R11, 35A15, 35B33}

\begin{document}

\begin{abstract}
We prove the existence of infinitely many nontrivial weak periodic solutions for a class of fractional Kirchhoff problems driven by a relativistic Schr\"odinger operator with periodic boundary conditions and involving different types of nonlinearities.  
\end{abstract}

\maketitle

\section{\bf Introduction}

\noindent
In this paper we deal with the existence of infinitely many nontrivial weak periodic solutions for the following class of fractional Kirchhoff problems:
\begin{equation}\label{P}
\left\{
\begin{array}{ll}
(a+b|u|^{2}_{\h})(-\Delta+m^{2})^{s}u = \lambda |u|^{q-2}u+|u|^{p-2}u \, &\mbox{ in } \, (-\pi, \pi)^{3},\\
u(x+2\pi e_{i})=u(x) \, &\mbox{ for all } \, x\in\R^{3}, i=1, 2,3
\end{array}
\right.
\end{equation}
where $a, b, m>0$ are real positive numbers, $\{e_{i}\}_{i=1,2,3}$ is the canonical basis in $\R^{3}$, $\lambda>0$ is a parameter, 
$s\in (3/4, 1)$, and $1<q<2<p\leq \2$, where $\2=\frac{6}{3-2s}$ is the fractional critical exponent.
The fractional Schr\"odinger operator $(-\Delta+m^{2})^{s}$ is a nonlocal operator which can be defined for any $u=\sum_{k\in \Z^{3}} c_{k} \frac{e^{\imath k\cdot x}}{(2\pi)^{\frac{3}{2}}}\in \mathcal{C}^{\infty}_{2\pi}(\R^{3})$, that is $u$ is infinitely differentiable in $\R^{3}$ and $2\pi$-periodic in each variable, by 
\begin{equation}\label{nfrls}
(-\Delta+m^{2})^{s} u(x)=\sum_{k\in \Z^{3}} c_{k} (|k|^{2}+m^{2})^{s} \, \frac{e^{\imath k\cdot x}}{(2\pi)^{\frac{3}{2}}}
\end{equation}
where
$$
c_{k}:=\frac{1}{(2\pi)^{\frac{3}{2}}} \int_{(-\pi,\pi)^{3}} u(x)e^{- \imath k \cdot x}dx \quad (k\in \Z^{3})
$$
are the Fourier coefficients of the function $u$.
This operator can be extended by density on the Hilbert space
$$
\h:=\Bigl\{u=\sum_{k\in \Z^{3}} c_{k} \frac{e^{\imath k\cdot x}}{(2\pi)^{\frac{3}{2}}}\in L^{2}(-\pi,\pi)^{3}: \sum_{k\in \Z^{3}} (|k|^{2}+m^{2})^{s} \, |c_{k}|^{2}<+\infty \Bigr\}
$$
endowed with the norm
$
|u|_{\h}:=\left(\sum_{k\in \Z^{3}} (|k|^{2}+m^{2})^{s} |c_{k}|^{2}\right)^{1/2};
$
see \cite{A2, A3}.
\smallskip

When $m=0$, the operator in (\ref{nfrls}) appears in the study of quasi-geostrophic equations; see \cite{CC, KNV}. 
In $\R^{N}$, one of the main reasons of studying the operator $(-\Delta+m^{2})^{s}$ is related to its physical meaning.
Indeed, when $s= {1}/{2}$, the operator $(-\Delta+m^{2})^{{1}/{2}}-m$ corresponds to the free Hamiltonian of a relativistic particle with mass $m$; see \cite{LL}.
There is also a deep connection between $-[(-\Delta+m^{2})^{s}-m^{2s}]$ and the theory of L\'{e}vy processes: such operator is the infinitesimal generator of a relativistic $2s$-stable process $\{X^{m}_{t}\}_{t \geq 0}$, that is a L\'{e}vy process with characteristic function given by
\begin{equation*}\label{carfun}
\mathbb{E}(e^{i \xi \cdot X^{m}_{t}}) := e^{-t [(m^{2}+|\xi|^{2})^{s}-m^{2s}]}  \quad (\xi \in \R^{N});
\end{equation*}
see \cite{Bog} for more details. 
\noindent
More generally, fractional and nonlocal operators have achieved tremendous popularity in these years both pure and applied mathematical research. In fact, these operators arise in several concrete real-world applications such as phase transitions, flames propagation, chemical reaction in liquids, population dynamics, American options in finance, crystal dislocation. For more details and applications on this topic we refer to \cite{DPV, MBRS}.

On the other hand, if we take $s=1$ and we replace $\lambda |u|^{q-2}u+|u|^{\2-2}u$ by a more general nonlinearity $h(x, u)$, $(-\pi, \pi)^{3}$ by a bounded open set $\Omega$ and periodic boundary conditions by homogeneous Dirichlet boundary conditions, then \eqref{P} boils down to the well-known classical Kirchhoff equation 
\begin{equation}\label{SKE}
\left\{
\begin{array}{ll}
-\left(a+b\int_{\Omega} |\nabla u|^{2}dx \right)\Delta u=h(x,u) & \mbox{ in } \Omega, \\
u=0 & \mbox{ on } \partial\Omega, 
\end{array}
\right.
\end{equation}
which is related to the stationary analogue of the Kirchhoff equation
\begin{equation}\label{KE}
\rho u_{tt} - \left( \frac{p_{0}}{h}+ \frac{E}{2L}\int_{0}^{L} |u_{x}|^{2} dx \right) u_{xx} =0,
\end{equation}
introduced by Kirchhoff \cite{Kir} in $1883$ as an extension of the classical D'Alembert's wave equation for describing  the transversal oscillations of a stretched string. The parameters in equation \eqref{KE} have the following meanings:  $L$ is the length of the string, $h$ is the area of the cross-section, $E$ is the young modulus (elastic modulus) of the material, $\rho$ is the mass density, and $p_{0}$ is the initial tension.
We refer to \cite{B, P} for the early classical studies dedicated to \eqref{KE}. We also note that nonlocal boundary value problems like \eqref{SKE} model several physical and biological systems, where $u$ describes a process which depends on the average of itself, as for example, the population density; see \cite{ACM, ChL}.
However, only after the Lions' work \cite{LionsK}, where a functional analysis approach was given to attack a general Kirchhoff equation in arbitrary dimension with external force term, problem \eqref{SKE} began to catch the attention of several mathematicians; see \cite{ACF, ACM, AF, FJ, HLP, LLT, PZ} and the references therein. We would like to mention \cite{H, HT} for some results concerning Kirchhoff problems in closed manifolds.

In the nonlocal framework, Fiscella and Valdinoci \cite{FV}  proposed for the first time a stationary fractional Kirchhoff variational model in a bounded domain $\Omega\subset \R^{N}$ with homogeneous Dirichlet boundary conditions and involving a critical nonlinearity:
\begin{align}\label{FKE}
\left\{
\begin{array}{ll}
M\left(\int_{\R^{N}}|(-\Delta)^{\frac{s}{2}}u|^{2}dx\right)(-\Delta)^{s}u=\lambda f(x, u)+|u|^{\2-2}u \quad &\mbox{ in } \Omega,\\
u=0 &\mbox{ in } \R^{N}\setminus \Omega, 
\end{array}
\right. 
\end{align}
where $(-\Delta)^{s}$ is the fractional Laplacian in $\R^{N}$, $M$ is a continuous Kirchhoff function whose model case is given by $M(t)=a+bt$, $f$ is a continuous subcritical nonlinearity and $\lambda$ is a parameter.
In their correction of the early (one-dimensional) model, the tension on the string, which classically has a ``nonlocal'' nature arising from the average of the kinetic energy $\frac{|u_{x}|^{2}}{2}$ on $[0, L]$, possesses a further nonlocal behavior provided by the $H^{s}$-norm (or other more general fractional norms) of the function $u$; see \cite{FV} for more details.
After the pioneering work \cite{FV}, several existence and multiplicity have been established by different authors; see  \cite{AI1, AI2, AFP, FMBS, Ny, PXZ} and the references therein. We note that the Kirchhoff problems considered in literature are set in $\R^{N}$ or in bounded domains with homogeneous boundary conditions.
On the other hand, in these years the existence and multiplicity of periodic solutions for subcritical and critical problems have been studied in \cite{A2, A3, A4, A5, AMMB} and the references therein.

Motivated by the interest shared by the mathematical community towards fractional Kirchhoff problems and periodic solutions for elliptic equations, the goal of this paper is to investigate the existence of nontrivial periodic solutions for fractional Kirchhoff problems with subcritical or critical growth. 
Now, we state our main results concerning \eqref{P}.
\begin{thm}\label{thm1}
Let $2<p<4$. Then, problem \eqref{P} admits infinitely many periodic solutions, for all $\lambda>0$.   
\end{thm}

\begin{thm}\label{thm2}
Let $4<p<\2$. Then, there exists $\lambda_{*}>0$ such that, for all $\lambda\in (0, \lambda_{*})$,  problem \eqref{P} has infinitely many periodic solutions.   
\end{thm}

\begin{thm}\label{thm3}
Let $p=4$ and $b>\frac{1}{S^{2}_{4}}$. Then, problem \eqref{P} has infinitely many periodic solutions,  for all $\lambda>0$.   
\end{thm}

\begin{thm}\label{thm4}
Let $p=\2$. Then, there exists $\lambda_{*}>0$ such that, for all $\lambda\in (0, \lambda_{*})$,  problem \eqref{P} has infinitely many periodic solutions.  
\end{thm}

\noindent
The proofs of our main results are obtained by applying variational and topological arguments inspired by \cite{ACF, AzP, FJ} after transforming \eqref{P} into a local problem in one more dimension, via a suitable variant of the extension method \cite{CS} in the periodic setting. In this way, we are able to overcome the nonlocal nature of the operator $(-\Delta+m^{2})^{s}$.
More precisely, as proved in \cite{A2, A3} (see also \cite{ST}), for any $u\in \h$ there exists a unique function $U\in \X$, called periodic extension of $u$, weakly solving
\begin{equation*}
\left\{
\begin{array}{ll}
-\dive(y^{1-2s} \nabla U)+m^{2}y^{1-2s}U =0 &\mbox{ in }\mathcal{S}_{2\pi}:=(-\pi,\pi)^{3} \times (0,\infty),  \\
U_{| {\{x_{i}=0\}}}= U_{| {\{x_{i}=2\pi\}}} & \mbox{ on } \partial_{L}\mathcal{S}_{2\pi}:=\partial (-\pi,\pi)^{3} \times [0,\infty), \\
U(x,0)=u(x)  &\mbox{ on }  \partial^{0}\mathcal{S}_{2\pi}:=(-\pi,\pi)^{3} \times \{0\},
\end{array}
\right.
\end{equation*}
where $\X$ is defined as the closure of the set $\mathcal{C}^{\infty}_{2\pi}(\overline{\R^{3+1}_{+}})$ of smooth and $2\pi$-periodic (in $x$) functions in $\R^{3+1}_{+}$ with respect to the norm
$$
\|U\|_{\X}:=\Bigl(\iint_{\mathcal{S}_{2\pi}} y^{1-2s} (|\nabla U|^{2}+m^{2s} U^{2}) \, dx dy\Bigr)^{1/2}.
$$  
Moreover, $\h$ coincides with the space of traces on $(-\pi,\pi)^{3} \times \{0\}$ of functions in $\X$.\\
The importance of the extension function $U$ is that it is connected to the operator \eqref{nfrls} of the original function $u$ through the formula  
$$
-\lim_{y\rightarrow 0^{+}} y^{1-2s} \frac{\partial U}{\partial y}(x,y) = \kappa_{s} (-\Delta + m^{2})^{s} u(x) \mbox{ in } \mathbb{H}^{-s}_{2\pi}, 
$$
where $\displaystyle{\kappa_{s}= 2^{1-2s} \frac{\Gamma(1-s)}{\Gamma(s)}}$ and $\mathbb{H}^{-s}_{2\pi}$ denotes the dual space of $\h$;  see Section \ref{Section2}.

As a consequence of the previous facts, the study of \eqref{P} is equivalent to investigate  
the following degenerate elliptic problem 
\begin{equation}\label{R}
\left\{
\begin{array}{ll}
-\dive(y^{1-2s} \nabla U)+m^{2}y^{1-2s}U =0 &\mbox{ in }\mathcal{S}_{2\pi},  \\
\smallskip
U_{| {\{x_{i}=-\pi\}}}= U_{| {\{x_{i}=\pi\}}} & \mbox{ on } \partial_{L}\mathcal{S}_{2\pi}, \\
\smallskip
(a+b\|U\|^{2}_{\X}){\partial_{\nu}^{1-2s} U}=\kappa_{s} [\lambda|u|^{q-2}u+|u|^{p-2}u]   &\mbox{ on }\partial^{0}\mathcal{S}_{2\pi},
\end{array}
\right.
\end{equation}
where
$$
{\partial_{\nu}^{1-2s} U(x)}:=-\lim_{y \rightarrow 0^{+}} y^{1-2s} \frac{\partial U}{\partial y}(x,y)
$$
is the conormal exterior derivative of $U$.

In the light of the variational structure of $(\ref{R})$, we look for critical points of the energy functional $\mathcal{J}: \X\rightarrow \R$ given by
 $$
 \mathcal{J}(U)=\frac{a}{2} \|U\|_{\X}^{2}+\frac{b}{4} \|U\|_{\X}^{4}-\lambda \frac{\kappa_{s}}{q}|\T(U)|_{q}^{q} -\frac{\kappa_{s}}{p} |\T(U)|_{p}^{p}.
 $$
Clearly, the functions $U\in \X$ are defined in $\A$, so the local arguments that we will develop along the paper have to be handled carefully in order to take care of the trace of the involved functions.
For instance, in the proof of Theorem \ref{thm1}, we use Clark's Theorem \cite{Clark} after proving a very interesting relation for appropriate finite dimensional spaces of $\X$; see Sections \ref{Section3}. Again, in the critical case, in order to recover some compactness properties for the functional $\J$, we establish a suitable variant of the concentration-compactness lemma \cite{Lions} in periodic setting taking into account that $\X$ is locally compactly embedded in the weighted Lebesgue spaces $L^{2}(\A, y^{1-2s})$; see  Section \ref{Section5}. We also stress that the restriction $s\in (\frac{3}{4}, 1)$ plays an important role to apply our variational arguments. Indeed, when $s\in (0, \frac{3}{4}]$, since $2^{*}_{s}\leq 4$, it is difficult to derive the geometric structure of the functional and the boundedness, convergence of the Palais-Smale sequences.\\
To our knowledge, this is the first time that the existence and multiplicity of periodic solutions for fractional Kirchhoff problems are investigated in the literature. Moreover, we believe that our results can be improved and extended for more general fractional Kirchhoff equations in closed manifolds. These questions will be addressed in future works. 
\smallskip

\noindent
The structure of the paper is the following. In Section \ref{Section2} we present some preliminary results concerning the  fractional periodic Sobolev spaces and the extension method in periodic setting. We also recall some basic notions on the Krasnoselskii genus that we will use in the proof of our main results. In Sections \ref{Section3} we deal with periodic solutions in the cases $2<p<4$ and $p=4$. In Section \ref{Section4} we study the case $4<p<\2$ using an appropriate truncated functional inspired by \cite{AzP}. In Section \ref{Section5}, when $p=\2$, we prove the existence of infinitely many solutions using a truncated functional and a suitable version of the concentration-compactness lemma \cite{Lions}.

\section{\bf Preliminaries}\label{Section2}
\subsection{Extension method in periodic setting}
In this section we fix the notations and we collect some preliminary facts for future references. Throughout this paper, we denote the upper half-space in $\R^{4}$ by
$$
\R^{3+1}_{+}=\{(x,y)\in \R^{4}: x\in \R^{3}, y>0 \}.
$$
\noindent
Let $\mathcal{S}_{2\pi}:=(-\pi,\pi)^{3}\times(0,\infty)$ denote the half-cylinder in $\R^{3+1}_{+}$ with basis $\partial^{0}\mathcal{S}_{2\pi}:=(-\pi,\pi)^{3}\times \{0\}$ and lateral boundary $\partial_{L}\mathcal{S}_{2\pi}:=\partial (-\pi,\pi)^{3}\times [0,+\infty)$. We denote by $B_{r}^{+}(z,t)=\{(x,y)\in \R^{3+1}_{+}: |(x,y)-(z,t)|<r\}$ the ball in $\R^{3+1}_{+}$ of center $(z,t)\in \R^{3+1}_{+}$ and radius $r>0$.
We also use the notation $|u|_{r}$ to denote the norm of any function $u: \R^{3}\rightarrow \R$ in  $L^{r}(-\pi,\pi)^{3}$.\\
Now, we recall the it is possible to define a trace operator from $\X$ to $\h$.  
\begin{thm}\cite{A2, A3}\label{tracethm}
There exists a surjective linear operator $\textup{Tr} : \X\rightarrow \h$  such that:
\begin{itemize}
\item[$(i)$] $\textup{Tr}(U)=U|_{\partial^{0} \mathcal{S}_{2\pi}}$ for all $U\in \mathcal{C}_{2\pi}^{\infty}(\overline{\R^{3+1}_{+}}) \cap \X$;
\item[$(ii)$] $\textup{Tr}$ is bounded and
\begin{equation}\label{tracein}
\sqrt{\kappa_{s}} |\textup{Tr}(U)|_{\h}\leq \|U\|_{\X},
\end{equation}
for every $U\in \X$.
In particular, equality holds in \eqref{tracein} for some $v\in \X$ if and only if $v$ weakly solves the following equation
$$
-\dive(y^{1-2s} \nabla U)+m^{2}y^{1-2s}U =0 \, \mbox{ in } \, \mathcal{S}_{2\pi}.
$$
\end{itemize}
\end{thm}
\begin{thm}\cite{A2, A3}\label{compacttracethm}
Let  $1\leq q \leq \2$. Then $\textup{Tr}(\X)$ is continuously embedded in $L^{q}(-\pi,\pi)^{3}$.  Moreover,  $\textup{Tr}(\X)$ is compactly embedded in $L^{q}(-\pi,\pi)^{3}$  for any  $1\leq q < \2$.
\end{thm}

\noindent
Taking into account Theorem \ref{tracethm} and Theorem \ref{compacttracethm}, it is possible to introduce the notion of extension for a function $u\in \h$. More precisely, we have the following result:
\begin{thm}\cite{A2, A3}
Let $u\in \h$. Then, there exists a unique $U\in \X$ such that
\begin{equation*}\label{extPu}
\left\{
\begin{array}{ll}
-\dive(y^{1-2s} \nabla U)+m^{2}y^{1-2s}U =0 &\mbox{ in }\mathcal{S}_{2\pi},  \\
U_{| {\{x_{i}=-\pi\}}}= U_{| {\{x_{i}=\pi\}}} & \mbox{ on } \partial_{L}\mathcal{S}_{2\pi}, \\
U(\cdot, 0)=u  &\mbox{ on } \partial^{0}\mathcal{S}_{2\pi},
\end{array}
\right.
\end{equation*}
and
\begin{align*}\label{conormal}
-\lim_{y \rightarrow 0^{+}} y^{1-2s}\frac{\partial U}{\partial y}(x,y)=\kappa_{s} (-\Delta+m^{2})^{s}u(x) \mbox{ in } \mathbb{H}^{-s}_{2\pi}.
\end{align*}
We call $U\in \X$ the periodic extension of $u\in \h$, and we denote it by $\textup{Ext}(u)$. 
Moreover, $\textup{Ext}(u)$ has the following properties:
\begin{itemize}
\item[$(i)$] $\textup{Ext}(u)$ is smooth for $y>0$ and $2\pi$-periodic in $x$;
\item[$(ii)$] $\|\textup{Ext}(u)\|_{\X}\leq \|V\|_{\X}$ for any $V\in \X$ such that $\textup{Tr}(V)=u$;
\item[$(iii)$] $\|\textup{Ext}(u)\|_{\X}=\sqrt{\kappa_{s}} |u|_{\h}$.
\end{itemize}
\end{thm}
\noindent
In view of the previous facts, we can reformulate the nonlocal problem (\ref{P}) with periodic boundary conditions, in a local way according to the following definition.
\begin{defn}
We say that $u\in \h$ is a weak solution to \eqref{P} if and only if $u=\textup{Tr}(U)$
and $U\in \X$ satisfies
\begin{align*}
(a+b\|U\|^{2}_{\X}) & \iint_{\mathcal{S}_{2\pi}} y^{1-2s} (\nabla U \nabla V + m^{2} U V  ) \, dxdy \\
&=\kappa_{s} \int_{(-\pi, \pi)^{3}} [\lambda |\T(U)|^{q-2}\T(U)+|\T(U)|^{p-2}\T(U)]\T(V) \,dx
\end{align*}
for every $V\in \X$.
\end{defn}

\noindent
Next, we prove a very useful result which derives from the theory for the powers of closed positive self-adjoint unbounded operators in a Hilbert space:
\begin{lem}\label{lemmino}
\begin{itemize}
\item[$(i)$] The operator $(-\Delta+m^{2})^{s}$ has a countable family of eigenvalues $\{\lambda_{\ell}\}_{\ell\in \N}$ which can be written as an increasing sequence of positive numbers
$$
0<\lambda_{1}<\lambda_{2}\leq \dots\leq \lambda_{\ell}\leq \lambda_{\ell+1}\leq \dots
$$
Each eigenvalue is repeated a number of times equal to its multiplicity $\mathopen{(}$which is finite$\mathclose{)}$$;$
 \item[$(ii)$] $\lambda_{\ell}=\mu_{\ell}^{s}$ for all $\ell\in \N$, where $\{\mu_{\ell}\}_{\ell\in \N}$ is the increasing sequence of eigenvalues of $-\Delta+m^{2}$$;$
\item[$(iii)$] $\lambda_{1}=m^{2s}$ is simple, $\lambda_{\ell}=\mu_{\ell}^{s} \rightarrow +\infty$ as $\ell \rightarrow +\infty$$;$
 \item[$(iv)$] The sequence $\{u_{\ell}\}_{\ell\in \N}$ of eigenfunctions corresponding to $\lambda_{\ell}$ is an orthonormal basis of $L^{2}(-\pi, \pi)^{3}$ and an orthogonal basis of the Sobolev space $\h$.
Let us note that $\{u_{\ell}, \mu_{\ell}\}_{\ell\in \N}$ are the eigenfunctions and eigenvalues of $-\Delta+m^{2}$ under periodic boundary conditions$;$
\item[$(v)$] For any $\ell\in \N$, $\lambda_{\ell}$ has finite multiplicity, and there holds
$$
\lambda_{\ell}=\min_{u\in \mathbb{P}_{\ell}\setminus \{0\}} \frac{|u|^{2}_{\h}}{|u|^{2}_{2}}   \quad \mbox{$\mathopen{(}$Rayleigh's principle$\mathclose{)}$}
$$
where
$$
\mathbb{P}_{\ell}:=\{u\in \h: \langle u, u_{j} \rangle_{\h} =0, \mbox{ for } j=1, \dots, \ell-1\}.
$$
\end{itemize}
\end{lem}
\begin{proof}
It is enough to prove that 
$$
(-\Delta+m^{2})^{-s}:L^{2}(-\pi,\pi)^{3} \rightarrow L^{2}(-\pi,\pi)^{3}
$$
is a self-adjoint, positive  and compact operator.

Firstly, we observe that if $u=\sum_{k\in \Z^{3}} \frac{c_{k}}{(2\pi)^{\frac{3}{2}}} e^{\imath k\cdot x}$ and $v=\sum_{k\in \Z^{3}} \frac{d_{k}}{(2\pi)^{\frac{3}{2}}} e^{\imath  k\cdot x}$ belong to $L^{2}(-\pi,\pi)^{3}$, then
\begin{align*}
\langle(-\Delta+m^{2})^{-s}u, v\rangle_{L^{2}(-\pi, \pi)^{3}}&=\sum_{k\in \Z^{3}} \frac{c_{k}}{(|k|^{2}+m^{2})^{s}} \bar{d}_{k}\\
&=\sum_{k\in \Z^{N}} c_{k}\frac{\bar{d}_{k}}{(|k|^{2}+m^{2})^{s}}\\
&=\langle u, (-\Delta+m^{2})^{-s}v\rangle_{L^{2}(-\pi, \pi)^{3}}
\end{align*}
that is $(-\Delta+m^{2})^{-s}$ is self-adjoint.

Clearly, for $u=\sum_{k\in \Z^{3}} \frac{c_{k}}{(2\pi)^{\frac{3}{2}}} e^{\imath  k\cdot x}\in L^{2}(-\pi,\pi)^{3}$ we have
$$
\langle(-\Delta+m^{2})^{-s}u, u\rangle_{L^{2}(-\pi, \pi)^{3}}=\sum_{k\in \Z^{3}} \frac{|c_{k}|^{2}}{(|k|^{2}+m^{2})^{s}}\geq 0,
$$
and $\langle (-\Delta+m^{2})^{-s}u, u\rangle_{L^{2}(-\pi, \pi)^{3}}>0$ if $u\neq 0$.\\
Finally, we show that $(-\Delta+m^{2})^{-s}$ is compact. Let $\{v_{j}\}_{j\in \N}$ be a bounded sequence in $L^{2}(-\pi, \pi)^{3}$ and let us denote by $\{d^{j}_{k}\}_{k\in \Z^{3}}$ its Fourier coefficients. Since $L^{2}(-\pi, \pi)^{3}\subset \mathbb{H}_{2\pi}^{-s}$, it is clear that $\{v_{j}\}_{j\in \N}$ is bounded in $\mathbb{H}_{2\pi}^{-s}$.
Let us denote by $u_{j}=(-\Delta+m^{2})^{-s} v_{j}$. \\
Thus, 
$$ 
|u_{j}|^{2}_{\h}=\sum_{k\in \Z^{3}}(|k|^{2}+m^{2})^{s} \frac{|d^{j}_{k}|^{2}}{(|k|^{2}+m^{2})^{2s}}=|v_{j}|_{\mathbb{H}^{-s}_{T}}^{2},
$$
that is $\{u_{j}\}_{j\in \N}$ is a bounded sequence in $\h$. Using the compactness of $\h$ into $L^{2}(-\pi, \pi)^{3}$, we deduce that $\{u_{j}\}_{j\in \N}$ admits a convergent subsequence in $L^{2}(-\pi, \pi)^{3}$. Accordingly, $(-\Delta+m^{2})^{-s} v_{j}$ strongly converges in $L^{2}(-\pi, \pi)^{3}$.
\end{proof}

Next, we aim to find some useful relations between the eigenvalues $\{\lambda_{j}\}_{j\in\N}$ of $(-\Delta+m^{2})^{s}$ and the corresponding extended eigenvalue problem in the half-cylinder $\mathcal{S}_{2\pi}$,
\begin{equation}\label{ERP}
\left\{
\begin{array}{ll}
-\dive(y^{1-2s} \nabla V)+m^{2}y^{1-2s}V =0 &\mbox{ in }\A, \\
\smallskip
V_{| {\{x_{i}=0\}}}= V_{| {\{x_{i}=T\}}} & \mbox{ on } \partial_{L}\A,\\
\smallskip
{\partial_{\nu}^{1-2s} V}=\kappa_{s} \lambda_{j} \T(V)   &\mbox{ on }\partial^{0}\A.
\end{array}
\right.
\end{equation}
Let us introduce the following notations. Set
\begin{equation}\label{defvh}
\mathbb{V}_{h}:=Span\{V_{1}, \dots, V_{h}\},
\end{equation}
where every $V_{j}$ solves (\ref{ERP}).
Clearly, $\T(V_{j})=u_{j}$ for all $j\in \N$, where $\{u_{j}\}_{j\in\N}$ is the basis of eigenfunctions in $\h$, defined in Lemma \ref{lemmino}. For any $h\in \N$, we define
\begin{equation}\label{defvhort}
\mathbb{V}_{h}^{\perp}:=\{V\in \X: \langle V, V_{j} \rangle_{\X} =0, \mbox{ for } j=1, \dots, h\}.
\end{equation}
Since $V_{j}$ solves (\ref{ERP}), then we deduce that
$$
\mathbb{V}_{h}^{\perp}=\{V\in \X: \langle \T(V), \T(V_{j}) \rangle_{L^{2}(-\pi, \pi)^{3}} =0, \mbox{ for } j=1, \dots, h\}.
$$
Hence, $ \X=\mathbb{V}_{h}\bigoplus \mathbb{V}_{h}^{\perp}$. Let us stress that the trace operator is bijective on 
$$
E:=\{V\in \X: V \mbox{ solves } (\ref{extPu})\}.
$$
Indeed, if $\tilde{V}_{1}$ and $\tilde{V}_{2}$ are the extension of $\tilde{u}_{1},\tilde{u}_{2} \in \h$ respectively, then
\begin{equation}\label{*}
\langle \tilde{V}_{i}, \varPhi \rangle_{\X} = k_{s} \langle \tilde{u}_{i}, \T(\varPhi) \rangle_{\h} \quad \forall \varPhi \in \X, i=1,2.
\end{equation}
If $\tilde{u}_{1}=\T(\tilde{V}_{1})= \T(\tilde{V}_{2})=\tilde{u}_{2}$, it follows from (\ref{*}) that
$$
\langle \tilde{V}_{1}-\tilde{V}_{2}, \varPhi \rangle_{\X}=0 \quad \forall \varPhi \in \X,
$$
so we deduce that $\tilde{V}_{1}=\tilde{V}_{2}$, that is $\T$ is injective on $E$. From this and the linearity of the trace operator $\T$, we get
$$
\dim \mathbb{V}_{h}=\dim Span\{\T(V_{1}), \cdots, \T(V_{h})\}=h.
$$
Now we prove that $\|\cdot\|_{\X}$ and $|\cdot|_{2}$ are equivalent norms on the finite dimensional space $\mathbb{V}_{h}$. More precisely, for any $V\in \mathbb{V}_{h}$, it holds
\begin{align}\label{eqnorm}
\kappa_{s}m^{2s} |\T(V)|_{2}^{2}\leq \|V\|_{\mathbb{X}^{s}_{2\pi}}^{2}\leq \kappa_{s}\lambda_{h} |\T(V)|_{2}^{2}.
\end{align}
Firstly, we note that $\{V_{j}\}_{j\in \N}$ is an orthogonal system in $\X$, since $\{\T(V_{j})\}_{j\in \N}$ is an orthonormal system in $L^{2}(-\pi, \pi)^{3}$, and $V_{j}$ satisfies
$$
\langle Z, V_{j} \rangle_{\X}=\kappa_{s} \lambda_{j}\langle \T(Z), \T(V_{j}) \rangle_{L^{2}(-\pi, \pi)^{3}} \mbox{ for all } Z\in \X, j\in \N.
$$
Then, using the fact that $\{\lambda_{j}\}_{j\in \N}$ is an increasing sequence (see $(i)$ of Lemma \ref{lemmino}), and the trace inequality (\ref{tracein}), for any $V=\sum_{j=1}^{h} \alpha_{j} V_{j}\in \mathbb{V}_{h}$ we have
\begin{align*}
\kappa_{s}m^{2s} |\T(V)|_{2}^{2}&\leq \|V\|_{\mathbb{X}^{s}_{2\pi}}^{2}=\sum_{j=1}^{h} \alpha_{j}^{2}\|V_{j}\|^{2}_{\X}\\
&=\kappa_{s}\sum_{j=1}^{h} \lambda_{j} \alpha_{j}^{2}|\T(V_{j})|^{2}_{2}\leq \kappa_{s}\lambda_{h}\sum_{j=1}^{h} \alpha_{j}^{2}|\T(V_{j})|^{2}_{2}\\
&= \kappa_{s}\lambda_{h} |\T(V)|_{2}^{2}. 
\end{align*}
Finally, we prove that for any $V\in \mathbb{V}_{h}^{\perp}$ the following inequality holds true
$$
\lambda_{h+1} |\T(V)|_{2}^{2}\leq \frac{1}{\kappa_{s}}\|V\|^{2}_{\X}.
$$
Fix $V\in \mathbb{V}_{h}^{\perp}$. Then $\T(V)\in \mathbb{P}_{h+1}$.
Indeed $\T(V_{j})=u_{j}$ is a weak solution to $(-\Delta+m^{2})^{s}u=\lambda_{j} u$ and using the fact that
$$
\langle \T(V), \T(V_{j}) \rangle_{L^{2}(-\pi, \pi)^{3}}=0 \mbox{ for every } j=1, \dots, h,
$$
we can infer that $\langle \T(V), \T(V_{j}) \rangle_{\h}=0,$ for every $j=1, \dots, h$.
Hence, in the light of the variational characterization $(v)$ of Lemma \ref{lemmino} and the trace inequality (\ref{tracein}), we get
\begin{equation}\label{C}
\lambda_{h+1} |\T(V)|_{2}^{2}\leq  |\T(V)|^{2}_{\h}\leq \frac{1}{\kappa_{s}}\|V\|^{2}_{\X}.
\end{equation}

\begin{remark}
In order to lighten the notation, we will assume that $\kappa_{s}=1$, and, with abuse of notation, we denote the trace of a function $U:\R^{3+1}_{+}\rightarrow \R$ by $u$, that is $u=\T(U)$.
\end{remark}

\subsection{Krasnoselskii genus}
In this section we recall the definition of the genus and some of its fundamental properties which will be used along the paper; see \cite{AR, Rab} for more details.\\
Let $E$ be a real Banach space and let $\mathcal{U}$ denote the class of sets $A\subset E\setminus \{0\}$ such that $A$ is closed in $E$ and symmetric with respect to the origin, i.e. $u\in A$ implies $-u\in A$. For $A\in \mathcal{U}$, 
we define the genus $\gamma(A)$ of $A$ by the smallest integer $k$ such that there exists an odd continuous mapping from A to $\R^{k}\setminus \{0\}$. If there does not exist such a $k$, we put  $\gamma(A)=\infty$. 
Moreover, we set  $\gamma(\emptyset)=0$. 
Then we collect the following results.
\begin{prop}\label{prop3.1}
Let $A, B\in \mathcal{U}$. Then we have
\begin{compactenum}[(i)]
\item If there exists an odd continuous mapping from $A$ to $B$, then $\gamma(A)\leq \gamma(B)$.
\item If there is an odd homeomorphism between $A$ and $B$, then $\gamma(A)= \gamma(B)$.
\item If $\gamma(B)<\infty$, then $\gamma(\overline{A\setminus B})\geq \gamma(A)-\gamma(B)$.
\item If $\mathbb{S}^{n}$ is the sphere in $\R^{n}$, then $\gamma(\mathbb{S}^{n})=n+1$.
\item If $A$ is compact, then $\gamma(A)<+\infty$, and there exists $\delta>0$ such that $N_{\delta}(A)\in \mathcal{U}$ and $\gamma(N_{\delta}(A))=\gamma(A)$, where $N_{\delta}(A)=\{x\in E: \|x-A\|\leq \delta \}$.
\end{compactenum}
\end{prop}

\begin{thm}\label{Clarke}\cite{Clark}
Let $\J\in \mathcal{C}^{1}(X, \R)$ be a functional satisfying the Palais-Smale condition. Furthermore, let us suppose that
\begin{compactenum}[$(i)$]
\item $\J$ is bounded from below and even,
\item there is a compact set $K\in \mathcal{U}$ such that $\gamma(K)=k$ and $\sup_{x\in K} \J(x)<\J(0)$.
\end{compactenum}
Then, $\J$ possesses at least k pairs of distinct critical points and their corresponding critical values are less than $\J(0)$.
\end{thm}

\begin{prop}\label{prop3.5}
If $K\in \mathcal{U}$, $0\notin K$ and $\gamma(K)\geq 2$, then $K$ has infinitely many points.
\end{prop}

\section{\bf Proofs of Theorem \ref{thm1} and Theorem \ref{thm3}}\label{Section3}

\noindent
In this section we provide the proofs of Theorem \ref{thm1} and Theorem \ref{thm3}. Firstly, we consider the case $2<p<4$.
We  begin by proving some auxiliary lemmas.
\begin{lem}\label{lem4.1}
The functional $\J$ is bounded from below.
\end{lem}
\begin{proof}
Let us prove that $\J$ is coercive. Assume that $\{U_{n}\}_{n\in \N}\subset \X$ is a sequence such that $\|U_{n}\|_{\X}\rightarrow \infty$. Then, using Theorem \ref{compacttracethm}, we can see that
\begin{align*}
\J(U_{n})\geq \frac{b}{4}\|U_{n}\|^{4}_{\X}-\frac{\lambda}{qS_{q}^{q/2}}\|U_{n}\|^{q}_{\X}-\frac{\lambda}{pS_{p}^{p/2}}\|U_{n}\|^{p}_{\X},
\end{align*}
where
$$
S_{r}:=\inf_{U\in \X\setminus\{0\}} \frac{\|U\|^{2}_{\X}}{|u|_{r}^{2}} \mbox{ with } r\in \{p,q\}.
$$
Since $1<q<2<p<4$, we can conclude that $\J(U_{n})\ri \infty$.
\end{proof}

\begin{lem}\label{lem4.2}
The functional $\J$ satisfies the Palais-Smale condition at any level $c\in \R$.
\end{lem}
\begin{proof}
Let $\{U_{n}\}_{n\in \N}\subset \X$ be a sequence such that $\J(U_{n})\ri c$ and $\J'(U_{n})\ri 0$ as $n\ri \infty$.
In view of Lemma \ref{lem4.1}, we can deduce that $\{U_{n}\}_{n\in \N}$ is bounded in $\X$. Then, up to a subsequence, we may assume that
\begin{align}\begin{split}\label{4.1}
&U_{n}\rightharpoonup U \mbox{ in } \X  \\
&u_{n}\ri u \mbox{ in } L^{r}(-\pi, \pi)^{3} \quad \forall r\in [1, \2) \\
&u_{n}\ri u \mbox{ a.e. in } (-\pi, \pi)^{3}. 
\end{split}\end{align}
Using $\langle \J'(U_{n}), U_{n}-U\rangle=o_{n}(1)$, we have
\begin{align}\label{4.2}
&(a+b\|U_{n}\|^{2}_{\X}) \iint_{\A} y^{1-2s} [\nabla U_{n} \nabla (U_{n}-U)+m^{2}U_{n}(U_{n}-U)]\, dx dy \nonumber \\
&=\lambda \int_{(-\pi, \pi)^{3}} |u_{n}|^{q-2}u_{n}(u_{n}-u)\, dx+\int_{(-\pi, \pi)^{3}} |u_{n}|^{p-2}u_{n}(u_{n}-u)\, dx+o_{n}(1).
\end{align}
By \eqref{4.1} it follows that
\begin{align*}
\int_{(-\pi, \pi)^{3}} |u_{n}|^{q-2}u_{n}(u_{n}-u)\, dx=\int_{(-\pi, \pi)^{3}} |u_{n}|^{p-2}u_{n}(u_{n}-u)\, dx=o_{n}(1)
\end{align*}
which together with \eqref{4.2} implies that 
\begin{align*}
(a+b\|U_{n}\|^{2}_{\X}) \iint_{\A} y^{1-2s} [\nabla U_{n} \nabla (U_{n}-U)+m^{2}U_{n}(U_{n}-U)]\, dx dy=o_{n}(1).
\end{align*}
Since $\{U_{n}\}_{n\in \N}$ is bounded in $\X$, we obtain
$$
a\leq (a+b\|U_{n}\|^{2}_{\X})\leq a+bC \quad \forall n\in \N,
$$
so that 
\begin{align*}
\iint_{\A} y^{1-2s} [\nabla U_{n} \nabla (U_{n}-U)+m^{2}U_{n}(U_{n}-U)]\, dx dy=o_{n}(1).
\end{align*}
Taking into account this fact, \eqref{4.1} and that $\X$ is a Hilbert space, we get the thesis.
\end{proof}

\begin{proof}[Proof of Theorem \ref{thm1}]
For each $k\in \N$, we consider $\mathbb{V}_{k}=span\{V_{1},\dots, V_{k}\}$ defined as in \eqref{defvh}. In view of \eqref{eqnorm} and that all norm of $L^{r}(-\pi, \pi)^{3}$ are equivalents on the finite dimensional spaces, we deduce that the norms of $\X$ and $L^{q}(-\pi, \pi)^{3}$ are equivalents on $\mathbb{V}_{k}$. Hence, there exists a positive constant $C_{k}$ (depending on $k$) such that
\begin{align}
C_{k}\|U\|^{q}_{\X}\leq |u|_{q}^{q} \quad \forall U\in \mathbb{V}_{k}.
\end{align}
Then, for all $R>0$ and $U\in \mathbb{V}_{k}$ with $\|U\|_{\X}\in [0, R]$, we have 
$$
\J(U)\leq C\|U\|^{2}_{\X}-\frac{\lambda}{q}C_{k} \|U\|^{q}_{\X},
$$
where $C:=\frac{a}{2}+\frac{b}{4}R^{2}>0$.
Take $R>0$ such that $CR^{2-q}<\frac{\lambda}{q} C_{k}$, and we set $\mathbb{S}:=\{U\in \mathbb{V}_{k}: \|U\|_{\X}=r\}$, where $0<r<R$. Consequently, for all $U\in \mathbb{S}$, we find
\begin{align*}
\J(U)\leq Cr^{2}-\frac{\lambda}{q}C_{k}r^{q}<R^{q}\left[CR^{2-q}-\frac{\lambda}{q}C_{k}\right]<0=\J(0).
\end{align*}
Since $\mathbb{V}_{k}$ and $\R^{k}$ are isomorphic and $\mathbb{S}$ and $\mathbb{S}^{k-1}$ are homeomorphic, we can deduce that $\gamma(S)=k$.
Moreover, $\J$ is even, so, invoking Theorem \ref{Clarke}, we get at least $k$ pairs of different critical points. From the arbitrariness of $k$, we obtain infinitely many critical points of $\J$.
\end{proof}

\noindent
Now, we assume that $p=4$. Since the proof of the next result can be obtained arguing as in Theorem \ref{thm1}, we only give a sketch of the idea.
\begin{proof}[Proof of Theorem \ref{thm3}]
It is easy to show that  $\J$ is coercive. Indeed, if $\{U_{n}\}_{n\in \N}\subset \X$ is a sequence such that $\|U_{n}\|_{\X}\ri \infty$, we can use $b>\frac{1}{S_{4}^{2}}$ and $1<q<2$ to see that
\begin{align}
\J(U_{n})\geq \frac{1}{4}\left(b-\frac{1}{S_{4}^{2}}\right)\|U_{n}\|^{4}_{\X}-\frac{\lambda}{qS_{q}^{q/2}}\|U_{n}\|^{q}_{\X}\ri \infty.
\end{align} 
Then, we can argue as in Lemma \ref{lem4.2} and Theorem \ref{thm1} to infer  that $\J$ admits infinitely many critical points in $\X$. This ends the proof of Theorem \ref{thm3}.
\end{proof}

\section{\bf Proof of Theorem \ref{thm2}}\label{Section4}
In this section we deal with the case $4<p<\2$. Since $\J$ can not be bounded from below, we overcome this difficulty borrowing some ideas introduced in \cite{AzP}.\\
By Theorem \ref{compacttracethm}, we can note that
\begin{align}\label{5.1}
\J(U)\geq \frac{a}{2}\|U\|^{2}_{\X}-\frac{\lambda}{qS_{q}^{q/2}}\|U\|^{q}_{\X}-\frac{1}{pS_{p}^{p/2}}\|U\|^{p}_{\X}=:g(\|U\|^{2}_{\X}),
\end{align}
where
$$
g(t)=\frac{a}{2}t-\frac{\lambda}{qS_{q}^{q/2}}t^{q/2}-\frac{1}{pS_{p}^{p/2}}t^{p/2}.
$$
Hence, there exists $\lambda_{*}>0$ such that, if $\lambda\in (0, \lambda_{*})$, then $g$ achieves its positive maximum.
Assume that $\lambda\in (0, \lambda_{*})$. Let $R_{0}$ and $R_{1}$  be the first and second roots of $g$. Let $\phi\in \mathcal{C}^{\infty}_{c}([0, \infty))$ be such that $0\leq \phi\leq 1$, $\phi=1$ in $[0, R_0]$ and $\phi=0$ in $[R_1, \infty)$. Now, we consider the following truncated functional
$$
\I(U)=\frac{a}{2}\|U\|^{2}_{\X}+\frac{b}{4}\|U\|^{4}_{\X}-\frac{\lambda}{q}|u|^{q}_{q}-\phi(\|U\|^{2}_{\X}) \frac{1}{p}|u|^{p}_{p}.
$$
As in \eqref{5.1}, we can deduce that $\I(U)\geq \bar{g}(\|U\|^{2}_{\X})$ where
$$
\bar{g}(t)=\frac{a}{2}t-\frac{\lambda}{qS_{q}^{q/2}}t^{q/2}-\phi(t)\frac{1}{pS_{p}^{p/2}}t^{p/2}.
$$
Let us observe that $\I$ is coercive, even and satisfies the Palais-Smale condition at any level $c\in \R$. Moreover, arguing as in the proof of Theorem \ref{thm1} with $R<R_{0}$, we can see that
$$
\sup_{U\in \mathbb{S}} \I(U)=\sup_{U\in \mathbb{S}} \J(U)<0=\J(0)=\I(0),
$$
and then $\I$ has infinitely many critical points. Therefore, in order to prove Theorem \ref{thm3}, it is enough to verify that the critical points of $\I$ are the critical points of $\J$.
Note that each critical point $U$ of $\I$ fulfills $\bar{g}(\|U\|^{2}_{\X})\leq \I(U)<0$. Hence, $\|U\|^{2}_{\X}<R_{0}$ and $\I(U)=\J(U)$. On the other hand, from the continuity of $\I$, we can find a neighborhood $N$ of $U$ such that $\I(U)<0$ for all $U\in N$. Accordingly, $\I(U)=\J(U)<0$ for all $U\in N$, and we can deduce that $\I'(U)=\J'(U)=0$, that is $U$ is a critical point of $\J$.
 
\section{\bf Proof of Theorem \ref{thm4}}\label{Section5}
This last section is devoted to the critical case, that is when $p=\2$  in \eqref{P}. As in Section $3$, the functional $\J$ can not be bounded from below. Moreover, an extra difficulty arises in the study of \eqref{P} due to the lack of compactness of the embedding $\h$ in $L^{\2}(-\pi, \pi)^{3}$.
For this purpose, we give a variant of the concentration-compactness lemma in periodic framework; see \cite{BCPS, DMV, PP} for some related results.
In order to prove it, we first give the following definition:
\begin{defn}
We say that a sequence $\{U_{n}\}_{n\in \N}$ is tight in $\X$ if for every $\delta>0$ there exists $R>0$ such 
that $\int_{R}^{\infty} \int_{(-\pi, \pi)^{3}} y^{1-2s}(|\nabla U_{n}|^{2}+m^{2}U_{n}^{2}) \, dx dy\leq \delta$ for any $n\in \mathbb{N}$.
\end{defn}

\begin{lem}\label{CCL}
Let $\{U_{n}\}_{n\in \N}$ be a bounded tight sequence in $\X$, such that $U_{n}$ converges weakly to $U$ in $\X$.
Let $\mu, \nu$ be two non-negative measures on $\R^{3+1}_{+}$ and $\R^{3}$ respectively and such that
\begin{align}
&\lim_{n\rightarrow \infty} y^{1-2s} (|\nabla U_{n}|^{2}+m^{2}U_{n}^{2})= \mu \label{46FS1}
\end{align}
and
\begin{align}
&\lim_{n\rightarrow \infty} |u_{n}|^{\2}= \nu \label{46FS2}
\end{align}
in the sense of measures. Then, there exist an at most countable set $I$ and three families $\{x_{i}\}_{i\in I}\subset (-\pi, \pi)^{3}$, $\{\mu_{i}\}_{i\in I}$,  $\{\nu_{i}\}_{i\in I}\subset (0, \infty)$, with $\mu_{i}, \nu_{i}\geq 0$ for all $i\in I$, such that
\begin{align}
&\nu=|u|^{\2}+\sum_{i\in I} \nu_{i} \delta_{x_{i}} \label{47FS}\\
&\mu\geq y^{1-2s} (|\nabla U|^{2}+m^{2}U^{2})+\sum_{i\in I} \mu_{i} \delta_{x_{i}} \label{48FS} \\
&\mu_{i}\geq S_{*} \nu_{i}^{\frac{2}{2^{*}_{s}}} \quad \mbox{ for all }  i\in I.  \label{49FS} 
\end{align}
\end{lem}
\begin{proof}
We first suppose that $U\equiv 0$. We claim that for all $V\in \mathcal{C}^{\infty}_{2\pi}(\overline{\R^{3+1}_{+}})$ such that $\supp(V)$ is compact, it holds
\begin{equation}\label{DMV1}
C\left(\int_{(-\pi, \pi)^{3}} |v|^{\2} d\nu\right)^{\frac{2}{\2}}\leq  \iint_{\A} V^{2} \, d\mu
\end{equation}
for some constant $C>0$. For this purpose, we fix  $V\in \mathcal{C}^{\infty}_{2\pi}(\overline{\R^{3+1}_{+}})$ such that $K:=\supp(V)$ is compact. By Theorem \ref{compacttracethm} we know that
\begin{equation}\label{DMV2}
S_{*}  \left(\int_{(-\pi, \pi)^{3}} |vu_{n}|^{\2} dx \right)^{\frac{2}{\2}}\leq \iint_{\A} y^{1-2s} [|\nabla (VU_{n})|^{2}+m^{2}(VU_{n})^{2}] \, dx dy.
\end{equation}
Now, we note that \eqref{46FS2} implies that
\begin{align}\label{DMV3}
\int_{(-\pi, \pi)^{3}} |vu_{n}|^{\2} dx \ri \int_{(-\pi, \pi)^{3}}  |v|^{\2} d\nu.
\end{align}
On the other hand,
\begin{align}\label{DMV4}
&\iint_{\A} y^{1-2s} [|\nabla (VU_{n})|^{2}+m^{2}(VU_{n})^{2}] \, dx dy \nonumber \\
&=\iint_{\A} y^{1-2s} V^{2}[|\nabla U_{n}|^{2}+m^{2}U_{n}^{2}] \, dx dy +\iint_{\A} y^{1-2s} U_{n}^{2} |\nabla V|^{2} \, dx dy \nonumber \\
&+2 \iint_{\A} y^{1-2s} U_{n}V\nabla V \nabla U_{n}\, dx dy. 
\end{align}
Since $H^{1}(K, y^{1-2s})$ is compactly embedded in $L^{2}(K, y^{1-2s})$ (see \cite{DMV}), we have that $U_{n}\ri 0$ in $L^{2}(K, y^{1-2s})$ which yields
\begin{equation}\label{DMV5}
\iint_{\A} y^{1-2s} U_{n}^{2} |\nabla V|^{2} \, dx dy\leq C\iint_{K} y^{1-2s} U_{n}^{2}  \, dx dy \ri 0.
\end{equation}
Then, from H\"older's inequality, the boundedness of $\{U_{n}\}_{n\in \N}$ in $\X$ and \eqref{DMV5}, we get
\begin{align}\label{DMV6}
&\left| \iint_{\A} y^{1-2s} U_{n}V\nabla V \nabla U_{n}\, dx dy\right|  \nonumber \\
&\leq \left(\iint_{\A} y^{1-2s} U^{2}_{n}|\nabla V|^{2}\, dx dy\right)^{\frac{1}{2}} \left(\iint_{\A} y^{1-2s} V^{2} |\nabla U_{n}|^{2}\, dx dy\right)^{\frac{1}{2}} \nonumber\\
&\leq C  \left(\iint_{\A} y^{1-2s} U^{2}_{n} \, dx dy\right)^{\frac{1}{2}} \left(\iint_{\A} y^{1-2s} |\nabla U_{n}|^{2}\, dx dy\right)^{\frac{1}{2}} \nonumber\\
&\leq C\left(\iint_{K} y^{1-2s} U_{n}^{2}\, dx dy\right)^{\frac{1}{2}}\ri 0.
\end{align}
Now, taking into account \eqref{46FS1}, we have 
\begin{equation}\label{DMV7}
\iint_{\A} y^{1-2s} V^{2}(|\nabla U_{n}|^{2}+m^{2}U_{n}^{2}) \, dx dy\ri \iint_{\A}  V^{2} \, d\mu.
\end{equation}
Putting together \eqref{DMV2}-\eqref{DMV7}, we can conclude that \eqref{DMV1} holds true with $C=S_{*}$. Now, assume $U\not \equiv 0$. Then we can apply the above result to the sequence $W_{n}=U_{n}-U$. Indeed, if
\begin{align}
&\lim_{n\rightarrow \infty} y^{1-2s} (|\nabla W_{n}|^{2}+m^{2}W_{n}^{2})=\tilde{\mu} \label{46FS1w}\\
&\lim_{n\rightarrow \infty} |w_{n}|^{\2}=\tilde{\nu} \label{46FS2w},
\end{align}
in the sense of measures, then it holds
\begin{equation}
S_{*}\left(\int_{(-\pi, \pi)^{3}} |v|^{\2} d\tilde{\nu}\right)^{\frac{2}{\2}}\leq  \iint_{\A} y^{1-2s} (|\nabla V|^{2}+m^{2}V^{2}) \, d\tilde{\mu}
\end{equation}
 for all $V\in \mathcal{C}^{\infty}_{2\pi}(\overline{\R^{3+1}_{+}})$ such that $\supp(V)$ is compact.
By the Brezis-Lieb Lemma \cite{BL} we get
$$
|vw_{n}|_{\2}^{\2}=|vu_{n}|_{\2}^{\2}-|vu|_{\2}^{\2}+o_{n}(1),
$$
which in view of \eqref{46FS2} and \eqref{46FS2w} implies that
$$
\int_{(-\pi, \pi)^{3}} |v|^{\2} d\tilde{\nu}=\int_{(-\pi, \pi)^{3}} |v|^{\2} d\nu-\int_{(-\pi, \pi)^{3}} |vu|^{\2} dx, 
$$
that is
\begin{equation}\label{nu}
\nu=\tilde{\nu}+|u|^{\2}.
\end{equation}
On the other hand,
\begin{align*}
\iint_{\A} y^{1-2s} V^{2} (|\nabla U_{n}|^{2}+m^{2}U_{n}^{2}) \, dx dy &=\iint_{\A} y^{1-2s} V^{2} (|\nabla U|^{2}+m^{2}U^{2}) \, dx dy \nonumber \\
&\quad+\iint_{\A} y^{1-2s} V^{2} (|\nabla W_{n}|^{2}+m^{2}W_{n}^{2}) \, dx dy \nonumber \\
&\quad+2 \iint_{\A} y^{1-2s} V^{2}(\nabla W_{n} \nabla V+m^{2}W_{n}V) \, dx dy. 
\end{align*}
Letting $n\ri \infty$ and using \eqref{46FS1} and \eqref{46FS1w} we have
$$
\iint_{\A} V^{2}d\mu=\iint_{\A} y^{1-2s} V^{2} (|\nabla U|^{2}+m^{2}U^{2}) \, dx dy+\iint_{\A} V^{2}d\tilde{\mu},
$$
that is
\begin{equation}\label{mu}
\mu=\tilde{\mu}+y^{1-2s} (|\nabla U|^{2}+m^{2}U^{2}).
\end{equation}
Therefore, invoking \cite{Lions}, there exist an at most countable set $I$, a family of distinct points $\{x_{i}\}_{i\in I}\subset (-\pi, \pi)^{3}$ and $\{\mu_{i}\}_{i\in I},  \{\nu_{i}\}_{i\in I}\subset (0, \infty)$ with $\mu_{i}, \nu_{i}\geq 0$ for all $i\in I$ such that
\begin{align*}
\tilde{\nu}=\sum_{i\in I} \nu_{i} \delta_{x_{i}} \quad \tilde{\mu}\geq \sum_{i\in I} \mu_{i} \delta_{x_{i}} 
\end{align*}
and $\tilde{\mu}_{i}\geq S_{*} \tilde{\nu}_{i}^{\frac{2}{2^{*}_{s}}}$  for all $i\in I$. Taking into account \eqref{nu} and \eqref{mu} we get the desired result.

\end{proof}

\noindent
Arguing as in Section $4$, we consider the truncated functional $\I$ which is bounded from below. Now, we verify that $\J$ fulfills a local Palais-Smale condition.
\begin{lem}\label{lem7.1}
Let $\{U_{n}\}_{n\in \N}\subset \X$ be a bounded sequence such that $\J(U_{n})\ri c$ and $\J'(U_{n})\ri 0$, where $c$ is such that
$$
c<\left(\frac{1}{\theta}-\frac{1}{\2} \right)(aS_{*})^{3/2s}- \left[ \frac{\left(\frac{1}{q}-\frac{1}{\theta} \right)|(-\pi, \pi)^{3}|^{\frac{\2-q}{\2}}}{\left(\frac{1}{\theta}-\frac{1}{\2} \right)} \right]^{\frac{\2}{\2-q}} \left[\left(\frac{q}{\2}\right)^{\frac{\2}{\2-q}}-\left(\frac{q}{\2}\right)^{\frac{q}{\2-q}} \right]\lambda^{\frac{\2}{\2-q}}
$$
and $\theta\in (4, \2)$ is a fixed constant.
Then, there exists $\lambda_{*}>0$ such that, for all $\lambda\in (0, \lambda_{*})$, $\{U_{n}\}_{n\in \N}$ admits a convergent subsequence in $\X$.
\end{lem}

\begin{proof}
Our claim is to apply Lemma \ref{CCL}. Firstly, we prove that $\{U_{n}\}_{n\in \N}$ is a tight sequence. 

Assume by contradiction that there exists $\delta_{0}>0$ such that, for any $R>0$ one has, up to a subsequence, 
\begin{equation}\label{newter}
\int_{R}^{\infty} \int_{(-\pi, \pi)^{3}} y^{1-2s} ( |\nabla U_{n}|^{2} +m^{2}|U_{n}|^{2}) \, dx dy >\delta_{0} \quad \forall n\in \mathbb{N}. 
\end{equation}
Fix $\e>0$, and let $r>0$ be such that  
$$
\int_{r}^{\infty} \int_{(-\pi, \pi)^{3}} y^{1-2s} ( |\nabla U|^{2} +m^{2}|U|^{2}) \, dx dy< \e.  
$$
Let $j= [\frac{M}{\e}]$ be the integer part of $\frac{M}{\e}$, where $M$ is such that $\|U_{n}\|_{\X}^{2}\leq M$ for any $n\in \mathbb{N}$, and define $I_{k}= \{y \in \R_{+} \, : \, r+ k \leq y\leq r+k+1 \}$, $k=0, 1, \dots, j$. It is clear that 
\begin{align*}
\sum_{k=0}^{j} \int_{I_{k}} \int_{(-\pi, \pi)^{3}} y^{1-2s} ( |\nabla U_{n}|^{2} +m^{2}|U_{n}|^{2}) \, dx dy\leq \|U_{n}\|_{\X}^{2} \leq M < \e (j+1), 
\end{align*}
and then we can find $k_{0}\in \{0, \dots, j\}$ such that, up to a subsequence, 
\begin{align*}
\int_{I_{k_{0}}} \int_{(-\pi, \pi)^{3}} y^{1-2s} ( |\nabla U_{n}|^{2} +m^{2}|U_{n}|^{2}) \, dx dy \leq \e \quad \forall n\in \mathbb{N}. 
\end{align*}
Let $\psi(y)$ be a nondecreasing cut-off function in $[0, \infty)$ such that $\psi=0$ in $[0, r+k_{0}]$ and $\psi=1$ in $(r+k_{0}+1, \infty)$. Define $V_{n}(x, y)= U_{n}(x, y) \psi(y)$. Since $\langle \J'(U_{n}), V_{n}\rangle=o_{n}(1)$ and $V_{n}(x, 0)=0$, it follows that 
\begin{align*}
(a+b\|U_{n}\|^{2}_{\X}) &\iint_{\A} y^{1-2s} \psi(|\nabla U_{n}|^{2}+m^{2}U_{n}^{2})\, dx dy\\
&=-(a+b\|U_{n}\|^{2}_{\X})\iint_{\A} y^{1-2s} U_{n} \psi'\partial_{y} U_{n} \, dx dy+o_{n}(1).
\end{align*}
Using H\"older's inequality, the boundedness of $\{U_{n}\}_{n\in \mathbb{N}}$ in $\X$, and that $H^{1}((-\pi, \pi)^{3}\times I_{k_{0}}, y^{1-2s})\subset L^{2\gamma}((-\pi, \pi)^{3}\times I_{k_{0}},  y^{1-2s})$ where $\gamma=1+\frac{2}{3-2s}$ (see \cite{DMV}), we get
\begin{align*}
&\left| \int_{I_{k_{0}}} \int_{(-\pi, \pi)^{3}} y^{1-2s} U_{n} \psi'\partial_{y} U_{n} \, dx dy \right|\\
&\leq C\left(\iint_{\A} y^{1-2s} |\nabla U_{n}|^{2} \, dx dy  \right)^{\frac{1}{2}} \left(\int_{I_{k_{0}}} \int_{(-\pi, \pi)^{3}} y^{1-2s} U_{n}^{2} \, dx dy  \right)^{\frac{1}{2}} \\
&\leq C \left(\int_{I_{k_{0}}} \int_{(-\pi, \pi)^{3}} y^{1-2s} |U_{n}|^{2\gamma} \, dx dy  \right)^{\frac{1}{2\gamma}} \left(\int_{I_{k_{0}}} \int_{(-\pi, \pi)^{3}} y^{1-2s} \, dx dy  \right)^{\frac{1}{2\gamma'}} \\
&\leq C \left(\int_{I_{k_{0}}} \int_{(-\pi, \pi)^{3}}  y^{1-2s} (|\nabla U_{n}|^{2}+m^{2}U_{n}^{2}) \, dx dy  \right)^{\frac{1}{2}}<  C\sqrt{\e}. 
\end{align*}
In view of the above estimates, and noting that $a\leq (a+b\|U_{n}\|^{2}_{\X})\leq a+bM$ for all $n\in \mathbb{N}$, we have
\begin{align*}
a \int_{I_{k_{0}}} \int_{(-\pi, \pi)^{3}} y^{1-2s} ( |\nabla U_{n}|^{2} +m^{2}|U_{n}|^{2}) \, dxdy< (a+bM) C\sqrt{\e}
\end{align*}
which contradicts \eqref{newter}. Therefore, $\{U_{n}\}_{n\in \N}$ is tight in $\X$.

Now, since $\{U_{n}\}_{n\in \N}$ is bounded, we may assume that 
\begin{align}
&\lim_{n\rightarrow \infty} y^{1-2s} (|\nabla U_{n}|^{2}+m^{2}U_{n}^{2})=\mu \label{46FS1l}\\
&\lim_{n\rightarrow \infty} |u_{n}|^{\2}= \nu \label{46FS2l}
\end{align}
in the sense of measures, where $\mu$ and $\nu$ are two non-negative measures on $\R^{3+1}_{+}$ and $\R^{3}$, respectively. 
In the light of Lemma \ref{CCL}, there exist an at most countable set $I$, a family of distinct points $\{x_{i}\}_{i\in I}\subset (-\pi, \pi)^{3}$ and $\{\mu_{i}\}_{i\in I},  \{\nu_{i}\}_{i\in I}\subset (0, \infty)$, with $\mu_{i}, \nu_{i}\geq 0$ for all $i\in I$, such that
\begin{align}
&\nu=|u|^{\2}+\sum_{i\in I} \nu_{i} \delta_{x_{i}} \label{47FSl}\\
&\mu\geq y^{1-2s} (|\nabla U|^{2}+m^{2}U^{2})+\sum_{i\in I} \mu_{i} \delta_{x_{i}} \label{48FSl} \\
&\mu_{i}\geq S_{*} \nu_{i}^{\frac{2}{2^{*}_{s}}} \quad \forall i\in I \label{49FSl}. 
\end{align}
For all $\rho>0$, we consider $\Psi_{\rho}(x,y):=\Psi(\frac{x-x_{i}}{\rho}, \frac{y}{\rho})$, where $\Psi\in \mathcal{C}^{\infty}_{2\pi}(\overline{\R^{3+1}_{+}})$ is such that $0\leq \Psi\leq 1$, $\Psi=1$ in $B^{+}_{1}(0,0)$, $\Psi=0$ in $\A\setminus B^{+}_{2}(0,0)$ and $|\nabla \Psi_{\rho}|\leq \frac{C}{\rho}$. Since $\{\Psi_{\rho}U_{n}\}_{n\in \N}$ is bounded in $\X$, we get $\langle \J'(U_{n}), \Psi_{\rho}U_{n}\rangle=o_{n}(1)$, that is
\begin{align}\label{7.1FS}
(a+b\|U_{n}\|^{2}_{\X}) & \iint_{\A} y^{1-2s} U_{n} \nabla\Psi_{\rho}\nabla U_{n} \, dx dy\nonumber \\
&=- (a+b\|U_{n}\|^{2}_{\X})\iint_{\A} y^{1-2s} \Psi_{\rho}(|\nabla U_{n}|^{2}+m^{2}U_{n}^{2})\, dx dy \nonumber\\
&\quad + \lambda \int_{(-\pi, \pi)^{3}} |u_{n}|^{q}\psi_{\rho}\, dx+\int_{(-\pi, \pi)^{3}} \psi_{\rho}|u_{n}|^{\2}\, dx+o_{n}(1).
\end{align}
Assume that $\|U_{n}\|_{\X}\ri t_{0}$.
Since $\psi_{\rho}$ has compact support, we can see that
$$
\lim_{\rho\ri 0} \lim_{n\ri \infty}\int_{(-\pi, \pi)^{3}} |u_{n}|^{q}\psi_{\rho}\, dx=0.
$$ 
On the other hand, by H\"older's inequality, $\{U_{n}\}_{n\in \mathbb{N}}$ is bounded in $\X$, $U_{n}\rightarrow U$ in $L^{2}(K, y^{1-2s})$ for all compact set $K\subset \R^{3+1}_{+}$,
\begin{align*}
&\lim_{\rho\ri 0}\limsup_{n\ri \infty}\left| \iint_{\A} y^{1-2s} U_{n} \nabla\Psi_{\rho}\nabla U_{n} \, dx dy \right|\\
&=\lim_{\rho\ri 0}\limsup_{n\ri \infty}\left| \iint_{B^{+}_{2\rho}(x_{i},0)} y^{1-2s} U_{n} \nabla\Psi_{\rho}\nabla U_{n} \, dx dy \right|\\
&\leq \lim_{\rho\ri 0}\limsup_{n\ri \infty}  \left(\iint_{B^{+}_{2\rho}(x_{i},0)} y^{1-2s} |\nabla U_{n}|^{2} \, dx dy  \right)^{\frac{1}{2}} \left(\iint_{B^{+}_{2\rho}(x_{i},0)} y^{1-2s} U_{n}^{2}|\nabla \Psi_{\rho}|^{2} \, dx dy  \right)^{\frac{1}{2}} \\
&\leq \lim_{\rho\ri 0} \frac{C}{\rho} \left(\iint_{B^{+}_{2\rho}(x_{i},0)} y^{1-2s} U^{2} \, dx dy  \right)^{\frac{1}{2}} \\
&\leq \lim_{\rho\ri 0} \frac{C}{\rho} \left(\iint_{B^{+}_{2\rho}(x_{i},0)} y^{1-2s} |U|^{2\gamma} \, dx dy  \right)^{\frac{1}{2\gamma}}  \left(\iint_{B^{+}_{2\rho}(x_{i},0)} y^{1-2s} \, dx dy  \right)^{\frac{1}{2\gamma'}} \\
&\leq C \lim_{\rho\ri 0} \left(\iint_{B^{+}_{2\rho}(x_{i},0)} y^{1-2s} |U|^{2\gamma} \, dx dy  \right)^{\frac{1}{2\gamma}}= 0.
\end{align*}
Therefore, taking the limit as $n\ri \infty$ in \eqref{7.1FS}, by \eqref{46FS1l} and \eqref{46FS2l} we get
$$
(a+bt_{0}^{2}) \iint_{\A} \Psi_{\rho}\, d\mu\leq \int_{(-\pi, \pi)^{3}} \psi_{\rho} \, d\nu,
$$
and letting $\rho\ri 0$ we deduce that
\begin{align*}
\nu_{i}\geq (a+bt_{0}^{2})\mu_{i}\geq a\mu_{i}.
\end{align*}
In view of \eqref{49FSl}, $s\in (3/4,1)$ and $\theta\in (4, \2)$, we obtain that
\begin{align}\label{7.2FS}
\nu_{i}\geq (aS_{*})^{\frac{3}{2s}}\geq \left(\frac{1}{\theta}-\frac{1}{\2}\right)(aS_{*})^{\frac{3}{2s}}.
\end{align}
Now, we show that \eqref{7.2FS} can not occur, so that $I=\emptyset$. Assume by contradiction that \eqref{7.2FS} holds true for some $i\in I$.
Then, for all $n\in \mathbb{N}$, 
\begin{align*}
c&=\J(U_{n})-\frac{1}{\theta} \langle \J'(U_{n}), U_{n}\rangle+o_{n}(1) \\
&\geq -\lambda \left(\frac{1}{q}-\frac{1}{\theta}\right) \int_{(-\pi, \pi)^{3}}|u_{n}|^{q}\, dx+\left(\frac{1}{\theta}-\frac{1}{\2}\right) \int_{(-\pi, \pi)^{3}}|u_{n}|^{\2}\, dx+o_{n}(1) \\
&\geq -\lambda \left(\frac{1}{q}-\frac{1}{\theta}\right) \int_{(-\pi, \pi)^{3}}|u_{n}|^{q}\, dx+\left(\frac{1}{\theta}-\frac{1}{\2}\right) \int_{(-\pi, \pi)^{3}} \psi_{\rho} |u_{n}|^{\2}\, dx+o_{n}(1) \\
\end{align*}
and by passing to the limit as $n\ri \infty$ we have
\begin{align*}
c&\geq -\lambda \left(\frac{1}{q}-\frac{1}{\theta}\right) \int_{(-\pi, \pi)^{3}}|u|^{q}\, dx+\left(\frac{1}{\theta}-\frac{1}{\2}\right) \int_{(-\pi, \pi)^{3}} |u|^{\2}\, dx+\left(\frac{1}{\theta}-\frac{1}{\2}\right) \sum_{i\in I} \psi_{\rho}(x_{i})\nu_{i} \\
&=-\lambda \left(\frac{1}{q}-\frac{1}{\theta}\right) \int_{(-\pi, \pi)^{3}}|u|^{q}\, dx+\left(\frac{1}{\theta}-\frac{1}{\2}\right) \int_{(-\pi, \pi)^{3}} |u|^{\2}\, dx+\left(\frac{1}{\theta}-\frac{1}{\2}\right) \sum_{i\in I} \nu_{i} \\
&\geq  -\lambda \left(\frac{1}{q}-\frac{1}{\theta}\right) \int_{(-\pi, \pi)^{3}}|u|^{q}\, dx+\left(\frac{1}{\theta}-\frac{1}{\2}\right) \int_{(-\pi, \pi)^{3}} |u|^{\2}\, dx+\left(\frac{1}{\theta}-\frac{1}{\2}\right) (aS_{*})^{\frac{3}{2s}}.
\end{align*}
Applying H\"older's inequality we get
\begin{align*}
c&\geq \left(\frac{1}{\theta}-\frac{1}{\2}\right) (aS_{*})^{\frac{3}{2s}} -\lambda \left(\frac{1}{q}-\frac{1}{\theta}\right) |(-\pi, \pi)^{3}|^{\frac{\2-q}{\2}} \left(\int_{(-\pi, \pi)^{3}}|u|^{\2}\, dx  \right)^{\frac{q}{\2}}+ \left(\frac{1}{\theta}-\frac{1}{\2}\right) \int_{(-\pi, \pi)^{3}} |u|^{\2}\, dx.
\end{align*}
Now, observing that, for $t>0$, the function
$$
h(t)=\left(\frac{1}{\theta}-\frac{1}{\2}\right) t^{\2}-\lambda \left(\frac{1}{q}-\frac{1}{\theta}\right) |(-\pi, \pi)^{3}|^{\frac{\2-q}{\2}}t^{q}
$$
achieves its absolute minimum at the point
$$
t_{0}=\left[ \frac{q\lambda (\frac{1}{q}-\frac{1}{\theta}) |(-\pi, \pi)^{3}|^{\frac{\2-q}{\2}}}{\2\left( \frac{1}{\theta}-\frac{1}{\2} \right)} \right]^{\frac{1}{\2-q}}>0,
$$
we can deduce that
\begin{align*}
c&\geq \left(\frac{1}{\theta}-\frac{1}{\2} \right)(aS_{*})^{3/2s}- \left[ \frac{\left(\frac{1}{q}-\frac{1}{\theta} \right)|(-\pi, \pi)^{3}|^{\frac{\2-q}{\2}}}{\left(\frac{1}{\theta}-\frac{1}{\2} \right)} \right]^{\frac{\2}{\2-q}} \left[\left(\frac{q}{\2}\right)^{\frac{\2}{\2-q}}-\left(\frac{q}{\2}\right)^{\frac{q}{\2-q}} \right]\lambda^{\frac{\2}{\2-q}}
\end{align*}
which leads to a contradiction. Therefore, $I=\emptyset$ and $u_{n}\ri u$ in $L^{\2}(-\pi, \pi)^{3}$. Accordingly,
$$
\lim_{n\ri \infty} (a+b\|U_{n}\|^{2}_{\X}) \|U_{n}\|^{2}_{\X}=\lambda\int_{(-\pi, \pi)^{3}}|u|^{q}\, dx+\int_{(-\pi, \pi)^{3}}|u|^{\2}\,dx.
$$
On the other hand, from the weak convergence, we know that
$$
(a+bt_{0}^{2}) \|U\|^{2}_{\X}=\lambda\int_{(-\pi, \pi)^{3}}|u|^{q}\, dx+\int_{(-\pi, \pi)^{3}}|u|^{\2}\,dx
$$
which implies that
$$
(a+b\|U_{n}\|^{2}_{\X}) \|U_{n}\|^{2}_{\X}\ri  (a+bt_{0}^{2}) \|U\|^{2}_{\X}.
$$
Therefore, $ \|U_{n}\|^{2}_{\X}\ri \|U\|^{2}_{\X}$ as $n\ri \infty$, and recalling that $\X$ is a Hilbert space, we can conclude that $U_{n}\ri U$ in $\X$.
\end{proof}

\begin{remark}
By Lemma \ref{lem7.1}, we can deduce that 
$$
\left(\frac{1}{\theta}-\frac{1}{\2} \right)(aS_{*})^{3/2s}- \left[ \frac{\left(\frac{1}{q}-\frac{1}{\theta} \right)|(-\pi, \pi)^{3}|^{\frac{\2-q}{\2}}}{\left(\frac{1}{\theta}-\frac{1}{\2} \right)} \right]^{\frac{\2}{\2-q}} \left[\left(\frac{q}{\2}\right)^{\frac{\2}{\2-q}}-\left(\frac{q}{\2}\right)^{\frac{q}{\2-q}} \right]\lambda^{\frac{\2}{\2-q}}>0
$$
for $\lambda>0$ sufficiently small.
\end{remark}\label{Remark2}

\begin{lem}\label{lem7.2}
If $\I(U)<0$, then $\|U\|^{2}_{\X}<R_{0}$ and $\I(V)=\J(V)$ for all $V$ in a neighborhood of $U$. Moreover, $\I$ satisfies the Palais-Smale condition at any level $c<0$.
\end{lem}
\begin{proof}
Since $\bar{g}(\|U\|^{2}_{\X})\leq \I(U)<0$, we can argue as in the proof of Theorem \ref{thm2} to infer that $\I(V)=\J(V)$ for all $\|V\|_{\X}<\frac{R_{0}}{2}$. Moreover, if $\{U_{n}\}_{n\in \N}\subset \X$ is a sequence such that $\I(U_{n})\ri c<0$ and $\I'(U_{n})\ri 0$, then, for all $n$ sufficiently large, $\J(U_{n})=\I(U_{n})\ri c<0$ and $\J'(U_{n})=\I'(U_{n})\ri 0$. Due to the fact that $\I$ is coercive, we deduce that $\{U_{n}\}_{n\in \N}$ is bounded in $\X$. In virtue of Remark \ref{Remark2}, for $\lambda>0$ sufficiently small, we have
$$
c<0<\left(\frac{1}{\theta}-\frac{1}{\2} \right)(aS_{*})^{3/2s}- \left[ \frac{\left(\frac{1}{q}-\frac{1}{\theta} \right)|(-\pi, \pi)^{3}|^{\frac{\2-q}{\2}}}{\left(\frac{1}{\theta}-\frac{1}{\2} \right)} \right]^{\frac{\2}{\2-q}} \left[\left(\frac{q}{\2}\right)^{\frac{\2}{\2-q}}-\left(\frac{q}{\2}\right)^{\frac{q}{\2-q}} \right]\lambda^{\frac{\2}{\2-q}}
$$
so, by Lemma \ref{lem7.1}, it follows that, up to a subsequence, $U_{n}\ri U$ in $\X$.
\end{proof}

\begin{lem}\label{lem7.3}
Given $k\in \N$, there exists $\e=\e(k)>0$ such that $\gamma(\I^{-\e})\geq k$, where $\I^{-\e}=\{U\in \X: \I(U)\leq -\e\}$.
\end{lem}
\begin{proof}
Fix $k\in \N$ and let $\mathbb{V}_{k}$ be the set defined in \eqref{defvh}. Then there exists $C_{k}>0$ such that
$$
C_{k}\|U\|^{q}_{\X}\leq |u|^{q}_{q} \quad \forall U\in \mathbb{V}_{k}.
$$
Take $\bar{\rho}>0$ such that $\|U\|_{\X}=\bar{\rho}$ and $0<\|U\|^{2}_{\X}<R_{0}$, so that $\I(U)=\J(U)$. Proceeding as in the proof of Theorem \ref{thm1}, we can find $R>0$ such that $\J(U)<-\e$, for all $U\in \mathbb{V}_{k}$ with $U\in \mathbb{S}$. Hence, $\mathbb{S}\subset \I^{-\e}$, and noting that $ \I^{-\e}$ is a symmetric and closed set, by Proposition \ref{prop3.1} it follows that $\gamma(\I^{-\e})\geq \gamma(\mathbb{S})=k$.
\end{proof}

\noindent
Now, for all $k\in \N$, we define 
$$
\Gamma_{k}=\{C\subset \X: C \mbox{ is closed }, C=-C \mbox{ and } \gamma(C)\geq k\},
$$
$$
K_{c}=\{U\in \X: \I(U)=c \mbox{ and } \I'(U)=0\},
$$
and
$$
c_{k}=\inf_{C\in \Gamma_{k}} \sup_{U\in C} \I(u).
$$
Then we can prove that:
\begin{lem}\label{lem7.4}
Given $k\in \N$, the number $c_{k}$ is negative.
\end{lem}
\begin{proof}
By Lemma \ref{lem7.3}, for each $k\in \N$ there exists $\e=\e(k)>0$ such that $\gamma(\I^{-\e})\geq k$. Moreover, $0\notin \I^{-\e}$ and $\I^{-\e}\in \Gamma_{k}$. On the other hand, $\sup_{U\in \I^{-\e}}\I(U)\leq -\e$. Consequently,
$$
-\infty<c_{k}=\inf_{C\in \Gamma_{k}} \sup_{U\in C} \I(u)\leq \sup_{U\in \I^{-\e}} \I(u)\leq -\e<0.
$$
\end{proof}

\begin{lem}\label{lem7.5}
If $c=c_{k}=c_{k+1}=\dots=c_{k+r}$ for some $r\in \N$, then there exists $\lambda^{*}>0$ such that
$$
\gamma(K_{c})\geq r+1 \quad \forall \lambda\in (0, \lambda^{*}).
$$
\end{lem}

\begin{proof}
From $c=c_{k}=c_{k+1}=\dots = c_{k+r}$, Lemma \ref{lem7.1} and Lemma \ref{lem7.4}, it follows that $\I$ satisfies the Palais-Smale condition in $K_{c}$, and it is easy to see that $K_{c}$ is a compact set. Moreover, $K_{c}=-K_{c}$. 

If by contradiction $\gamma(K_{c})\leq r$, then there exists a closed and symmetric set $U$, with $K_{c}\subset U$, such that $\gamma(U)= \gamma(K_{c})\leq r$. Note that we can choose $U\subset \I^{0}$ thanks to $c<0$. By Theorem 3.4 in \cite{Benci}, we can find an odd homeomorphism $\eta: \X\rightarrow \X$ such that $\eta(\I^{c+\delta}- U)\subset \I^{c-\delta}$, for some $\delta\in (0, -c)$. Thus, $\I^{c+\delta}\subset \I^{0}$, and by definition of $c=c_{k+r}$, there exists $A\in \Gamma_{k+r}$ such that $\sup_{u\in A} \I(u)<c+\delta$, that is, $A\subset \I^{c+\delta}$, and 
\begin{align}\label{7.3}
\eta(A-U)\subset \eta(\I^{c+\delta} -U)\subset \I^{c-\delta}.
\end{align} 
However, by Proposition \ref{prop3.1}, $\gamma(\overline{A-U})\geq \gamma(A)-\gamma(U)\geq k$, and $\gamma(\eta(\overline{A-U}))\geq \gamma(\overline{A-U})\geq k$. Then, $\eta(\overline{A-U})\in \Gamma_{k}$ and this is impossible because of \eqref{7.3}.
\end{proof}
\noindent
In the light of the previous lemmas, we are ready to give the proof of the main result of this section.
\begin{proof}[Proof of Theorem \ref{thm4}]
If $-\infty<c_1<c_2<\dots<c_k<\dots<0$ and noting that each $c_k$ is a critical value of $\I$, we can obtain infinitely many critical points of $\I$, that is \eqref{R} admits infinitely many solutions in $\X$.
Moreover, if $c_k=c_{k+r}$, then $c=c_k=c_{k+1}=\dots=c_{k+r}$, and applying Lemma \ref{lem7.5} we can find $\lambda_{*}>0$ such that
$$
\gamma(K_c)\geq r+1\geq 2.
$$
By Proposition \ref{prop3.5}, $K_{c}$ has infinitely many points, that is \eqref{R} possesses infinitely many solutions.
\end{proof}

\noindent
{\bf Acknowledgments.}
The author would like to express his sincere gratitude to the referee for all insightful comments and valuable suggestions, which enabled to improve this version of the manuscript.

\end{document}